\documentclass[a4paper,reqno]{amsart}

\usepackage{amssymb, longtable}
\usepackage[mathcal]{euscript}

\usepackage{color}

\usepackage{hyperref}                  

\newcommand{\bydef}{:=}

\newcommand{\veps}{\varepsilon}

\newcommand{\id}{\mathrm{id}}

\newcommand{\cH}{\mathcal{H}}

\newcommand{\cL}{\mathcal{L}}

\newcommand{\frL}{{\mathfrak L}}
\newcommand{\frP}{{\mathfrak P}}

\newcommand{\ZZ}{\mathbb{Z}}

\newcommand{\PP}{\mathbb{P}}

\newcommand{\CC}{\mathbb{C}}

\newcommand{\FF}{\mathbb{F}}

\DeclareMathOperator{\Hom}{\mathrm{Hom}}

\DeclareMathOperator{\Aut}{\mathrm{Aut}}

\DeclareMathOperator{\Diag}{\mathrm{Diag}}

\DeclareMathOperator{\supp}{\mathrm{Supp}}

\newcommand{\frsl}{{\mathfrak{sl}}}

\newcommand{\frso}{{\mathfrak{so}}}

\numberwithin{equation}{section}

\newcommand{\subo}{_{\bar 0}}
\newcommand{\subuno}{_{\bar 1}}

\newcommand{\bup}{\textup{b}}

\newtheorem{theorem}{Theorem}[section]
\newtheorem{proposition}[theorem]{Proposition}
\newtheorem*{proposition*}{Proposition}
\newtheorem{lemma}[theorem]{Lemma}

\theoremstyle{definition}

\newtheorem{example}[theorem]{Example}

\theoremstyle{remark} \newtheorem{remark}[theorem]{Remark}

\begin{document}

\title{Non-group gradings on  simple Lie algebras}

\author{Alberto Elduque}
\address{Departamento de
Matem\'{a}ticas e Instituto Universitario de Matem\'aticas y
Aplicaciones, Universidad de Zaragoza, 50009 Zaragoza, Spain}
\email{elduque@unizar.es}
\thanks{${}^3$%
Supported by grants MTM2017-83506-C2-1-P (AEI/FEDER, UE) and E22\_17R
(Gobierno de Arag\'on, Grupo de referencia ``\'Algebra y
Geometr{\'\i}a'', cofunded by Feder 2014-2020 
``Construyendo Europa desde Arag\'on'')}



\date{July 21, 2021}

\begin{abstract}
A set grading on the split simple Lie algebra of type $D_{13}$, that cannot
be realized as a group-grading, is constructed by splitting the set of positive roots into a disjoint union of pairs of orthogonal roots, following a pattern provided by the lines of the projective plane over $GF(3)$. This answers in the negative \cite[Question 1.11]{EKmon}.

Similar non-group gradings are obtained for types $D_{n}$ with $n\equiv 1\pmod{12}$,
by substituting the lines in the projective plane by blocks of suitable Steiner systems.
\end{abstract}

\maketitle

\section{Introduction}\label{se:intro}

The aim of this paper is the construction of a set grading on a simple Lie algebra
that cannot be realized as a group grading, thus answering in the negative
\cite[Question 1.11]{EKmon}.

\smallskip

The systematic study of gradings on Lie algebras was initiated in 1989 by Patera and Zassenhaus \cite{PZ}, although particular gradings have been used from the beginning of Lie theory.

A \emph{set grading} on a Lie algebra $\cL$ over a field $\FF$ 
is a decomposition into a direct sum
of vector subspaces: $\Gamma:\cL=\bigoplus_{s \in S}\cL_s$, where $\cL_s\neq 0$ for any $s$ in the grading set $S$, such that for any $s_1,s_2\in S$, either 
$[\cL_{s_1},\cL_{s_2}]=0$ or there exists an element $s_3\in S$ such that
$[\cL_{s_1},\cL_{s_2}]\subseteq \cL_{s_3}$. The subspaces $\cL_s$ are called the
\emph{homogeneous components} of $\Gamma$.

Two such gradings $\Gamma:\cL=\bigoplus_{s \in S}\cL_s$ and
$\Gamma':\cL'=\bigoplus_{s' \in S'}\cL'_{s'}$ are said to be \emph{equivalent}
if there is an isomorphism of Lie algebras $\varphi:\cL\rightarrow\cL'$ such that
for any $s\in S$ there is an $s'\in S'$ such that $\varphi(\cL_s)=\cL'_{s'}$.

On the other hand, given a (semi)group $G$, a \emph{$G$-grading} on $\cL$ is
a decomposition as above $\Gamma:\cL=\bigoplus_{g\in G}\cL_g$, such that 
$[\cL_{g_1},\cL_{g_2}]\subseteq \cL_{g_1g_2}$ for all $g_1,g_2\in G$. The \emph{support} of $\Gamma$ is the subset 
$\supp(\Gamma)\bydef\{g\in G\mid \cL_g\neq 0\}$.

Given a set grading $\Gamma:\cL=\bigoplus_{s \in S}\cL_s$, it is said that
$\Gamma$ can be realized as a (semi)group grading, if there exists a (semi)group
$G$ and a one-to-one map $\iota: S\hookrightarrow G$ such that the subspaces
$\cL_{\iota(s)}\bydef\cL_s$, and $\cL_g\bydef 0$ if $g\not\in\iota(S)$, form a 
$G$-grading of $\cL$.

Given the set grading $\Gamma:\cL=\bigoplus_{s \in S}\cL_s$, its 
\emph{universal group} is the group defined by generators (the elements of $S$) and
relations as follows:
\[
U(\Gamma)\bydef\Bigl\langle S\mid s_1s_2=s_3\ \text{for all 
$s_1,s_2,s_3\in S$ with $0\neq [\cL_{s_1},\cL_{s_2}]\subseteq\cL_{s_3}$}\Bigr\rangle.
\]
There is a natural map $\iota:S\rightarrow U(\Gamma)$ taking $s$ to its coset
modulo the relations, and \emph{$\Gamma$ can be realized as a group grading if and only
if $\iota$ is one-to-one}.

The \emph{diagonal group} of $\Gamma$ is the subgroup $\Diag(\Gamma)$ 
of the group $\Aut(\cL)$
of automorphisms of $\cL$, consisting of those automorphisms that act by 
multiplication by a scalar on each homogeneous component. Any 
$\varphi\in\Diag(\Gamma)$ gives a map $\chi:S\rightarrow \FF^\times$ by the equation $\varphi\vert_{\cL_{s}}=\chi(s)\id$. This map induces a group homomorphism 
(a character with values in $\FF$), denoted by the same symbol,
$\chi:U(\Gamma)\rightarrow \FF^\times$. And conversely, any character $\chi$
determines a unique element in $\Diag(\Gamma)$. This shows that the diagonal group is isomorphic to the group of characters with values in 
$\FF$ of the universal group. If $\FF$ is algebraically closed of characteristic $0$ and $U(\Gamma)$ is abelian, characters separate points, and hence if $\Gamma$
is a group grading, the homogeneous components are the common
eigenspaces for the action of $\Diag(\Gamma)$.

For the basic results on gradings, the reader is referred to 
\cite[Chapter 1]{EKmon}.

\smallskip

In the seminal work \cite{PZ}, it was erroneously asserted that any set grading
on a finite-dimensional Lie algebra can be realized as a semigroup grading \cite[Theorem 1.(d)]{PZ}.
Counterexamples were given in \cite{Eld06} and \cite{Eld09}. In particular, in 
\cite{Eld09} there appears a non-semigroup grading on the semisimple Lie algebra 
$\frsl_2\oplus\frsl_2$. On the other hand, Cristina Draper proved  that if $G$ is
a semigroup and $\Gamma:\cL=\bigoplus_{g\in G}\cL_g$ is a $G$-grading on a 
simple Lie algebra $\cL$, then $\supp(\Gamma)$ generates an abelian subgroup 
of $G$ so, in particular, $\Gamma$ can be realized as a group grading. 
(See \cite[Proposition 1.12]{EKmon}.) As a direct consequence, \emph{the universal group
of any set grading on a finite-dimensional simple  Lie algebra is always abelian}.

\smallskip

There appeared naturally the following question \cite[Question 1.11]{EKmon}: 
\begin{center}
\emph{Can any set grading on a finite-dimensional simple Lie algebra over $\CC$\\
be realized as a group grading?}
\end{center}
The main goal of this paper is to answer this question in the negative.

\smallskip

The author is indebted to Alice Fialowski, who during the
``Workshop on Non-Associative Algebras and Applications'', held in Lancaster University in 2018,
showed him a 
$\left(\ZZ/2\right)^n$-grading on the orthogonal Lie algebra 
$\cL=\frso(2n)$ over $\CC$,
constructed by Andriy Panasyuk,
with some nice properties: one of the homogeneous component $\cL_g$, 
with $g\neq e$, is a Cartan subalgebra,
while all the other nonzero homogeneous components are toral two-dimensional
subalgebras. With hindsight, this grading
turns out to be the one in Example \ref{ex:Dn}. This is an example of
what Hesselink \cite{Hes} called \emph{pure gradings}. Some properties of these
pure gradings, useful for our purposes, are developed in Section \ref{se:pure}.

In an attempt to understand this grading on $\frso(2n)$ and to define similar gradings, 
the author used the $13$ lines in the projective plane over the field of 
three elements to split the set of positive roots of $\frso(26)$ (say, over $\CC$) into the disjoint union of pairs of orthogonal roots, thus obtaining the grading in \eqref{eq:D13}.
Checking that this is indeed a set grading is not obvious and relies on a pure
group grading on $\frso(8)$ in Example \ref{ex:D4}. The next step 
was to check whether it is a group grading, but surprisingly it turned out not to be the case (Theorem \ref{th:D13}). Moreover, the result can be extended to the
classical split simple Lie algebra of type $D_{13}$ over any arbitrary field
of characteristic not two (Theorem \ref{th:D13_2}). All this appears in Section 
\ref{se:D13}.

The last Section \ref{se:Steiner} is devoted to obtain similar non-group gradings on
the classical split simple Lie algebras of type $D_{2n}$ with 
$n\equiv 1\pmod{12}$, by
substituting the lines in the projective plane $\PP^2(\FF_3)$ by the blocks 
of Steiner systems of type $S(2,4,n)$.

\smallskip

Throughout the work, all the algebras considered will be assumed to be 
defined over a field $\FF$ and to be finite-dimensional. Up to Theorem \ref{th:D13},
the ground field $\FF$ will be assumed to be algebraically closed of characteristic $0$.

\smallskip

\section{Pure gradings on semisimple Lie algebras}\label{se:pure}

Let $\cL$ be a semisimple Lie algebra over an algebraically closed field $\FF$ of characteristic $0$. Following \cite{Hes}, a grading $\Gamma$ 
by the abelian group $G$: $\Gamma:\cL=\bigoplus_{g\in G}\cL_g$, is said to be
\emph{pure} if there exists a Cartan subalgebra $\cH$ of $\cL$ and an element 
$g\in G$, $g\neq e$, such that $\cH$ is contained in $\cL_g$. (Here $e$ denotes the neutral element of $G$.)

This section follows the ideas in \cite{Hes}.

Let $\Phi$ be the set of roots relative to $\cH$, so that 
$\cL=\cH\oplus\left(\bigoplus_{\alpha\in\Phi}\cL_\alpha\right)$ is the 
root space decomposition of $\cL$, and fix a system of simple roots
$\Delta=\{\alpha_1,\ldots,\alpha_n\}$ relative to the Cartan subalgebra $\cH$. Denote by $\Phi^+$ the set of positive roots relative to $\Delta$.

As in \cite[\S 4]{Hes}, consider the torus $T\bydef\{\varphi\in\Aut(\cL)\mid \varphi\vert_\cH=\id\}$ and its
$2$-torsion part: $T_2\bydef\{\varphi\in T\mid \varphi^2=\id\}$. Consider too the
larger subgroup $T_e\bydef\{\varphi\in\Aut(\cL)\mid \varphi\vert_\cH=\pm\id\}$.

\begin{lemma}\label{le:phiH-id}
Let $\sigma$ be an automorphism of $\cL$ with $\sigma\vert_\cH=-\id$. Then the following conditions hold:
\begin{enumerate}
\item 
The order of $\sigma$ is $2$: $\sigma^2=\id$, $\sigma$ permutes 
$\cL_\alpha$ and $\cL_{-\alpha}$ for all $\alpha\in\Phi$, and a system of generators $e_i\in\cL_{\alpha_i}$, $f_i\in\cL_{-\alpha_i}$, $i=1,\ldots,n$, can be chosen 
such that $[e_i,f_i]=h_i$,  
$\sigma(e_i)=-f_i$, and $\sigma(f_i)=-e_i$, where $h_i\in\cH$ satisfies $\alpha_i(h_i)=2$. 

\item The subgroup $T_e$ is the semidirect product of the torus $T$ and the cyclic
group of order $2$ generated by $\sigma$.

\item The centralizer of $\sigma$ in $T$ is $T_2$.
\end{enumerate}
\end{lemma}
\begin{proof}
This follows from \cite[Lemma 4.3]{Hes}. We include a proof for completeness.
From $\sigma\vert_\cH=-\id$ it follows at once that $\sigma$ permutes 
$\cL_\alpha$ and $\cL_{-\alpha}$ for all $\alpha\in\Phi$. For any 
$\alpha\in\Phi^+$, choose nonzero elements $e_\alpha\in\cL_\alpha$ 
and $f_\alpha\in\cL_{-\alpha}$ such that 
$[e_\alpha,f_\alpha]=h_{\alpha}$ with $\alpha(h_{\alpha})=2$. Then we have
$\sigma(h_\alpha)=-h_\alpha$, $\sigma(e_\alpha)=\mu f_{\alpha}$, 
and $\sigma(f_{\alpha})=\mu^{-1}e_\alpha$ for some 
$0\neq \mu\in\FF$. Let 
$\eta\in\FF$ be a square root of $-\mu^{-1}$. Then we get 
$\sigma(\eta e_\alpha)=-\eta^{-1}f_\alpha$ and 
$\sigma(\eta^{-1}f_\alpha)=-\eta e_\alpha$. This completes the proof of the first part.

For any $\varphi\in T_e$, either $\varphi\vert_\cH=\id$ and hence we have 
$\varphi\in T$, or $\varphi\vert_\cH=-\id$ and hence $\varphi\sigma\in T$. Therefore, $T_e$ equals $T\cup T\sigma$. Besides, $T$ is the kernel of
the natural homomorphism $T_e\rightarrow \{\pm 1\}$, so it is a normal subgroup
of $T_e$. This proves (2).

For any $\varphi\in T$, there are nonzero scalars $\mu_i\in\FF$ such that 
$\varphi(e_i)=\mu_ie_i$ and $\varphi(f_i)=\mu_i^{-1}f_i$ for all 
$i=1,\ldots,n$.
Then we get
\[
\begin{split}
\sigma\varphi(e_i)&=\mu_i\sigma(e_i)=-\mu_if_i,\\
\varphi\sigma(e_i)&=-\varphi(f_i)=-\mu_i^{-1}f_i,
\end{split}
\]
and hence, if $\varphi$ and $\sigma$ commute, we get $\mu_i=\mu_i^{-1}$, so that
$\mu_i^2=1$ for all $i$ and, therefore, $\varphi^2=\id$.
\end{proof}

The automorphism $\sigma$  in item (1) above play a key role in the proof of
the existence of Chevalley bases. (See \cite[\S 25.2]{Humphreys}.)

\begin{proposition}\label{pr:pure}
Let $\cL$ be a semisimple Lie algebra over an algebraically closed field $\FF$ of characteristic $0$. Let $G$ be an abelian group and let $\Gamma$ be a pure 
$G$-grading of $\cL$. Let $g\in G\setminus\{e\}$ such that $\cL_g$ contains 
a Cartan subalgebra $\cH$ of $\cL$. With $T$ and $T_2$ as before, there is an automorphism $\sigma\in\Aut(\cL)$ with 
$\sigma\vert_\cH=-\id$, such that 
$\Diag(\Gamma)$ is the cartesian product of $\Diag(\Gamma)\cap T_2$ and the
subgroup generated by $\sigma$:
\begin{equation}\label{eq:DiagG}
\Diag(\Gamma)=\left(\Diag(\Gamma)\cap T_2\right)\times\langle\sigma\rangle .
\end{equation}
In particular, $\Diag(\Gamma)$ is a finite $2$-elementary group.
\end{proposition}
\begin{proof}
Without loss of generality, we may assume that $G$ is the universal group 
$U(\Gamma)$. Then the group $\Diag(\Gamma)$ can be identified with the group of
characters of $G$.

Let $\chi$ be a character of $G$ with $\chi(g)\neq 1$. The automorphism $\sigma$ of $\Diag(\Gamma)$ 
defined by $\sigma(x)=\chi(h)x$ for any $h\in G$ and $x\in\cL_h$ satisfies that its restriction to $\cH$ is $\chi(g)\id$. This forces  
$\chi(g)^{-1}\alpha$ to be a root for any root $\alpha$ relative to 
$\cH$. As the only scalar multiples of a root $\alpha$ are $\pm\alpha$,
we get $\chi(g)=-1$, and hence $\sigma\vert_\cH=-\id$. This also shows
that $\chi(g)=\pm 1$ for any character
$\chi$ of $G$,  and hence $\Diag(\Gamma)$ is contained in $T_e$.
We conclude that $\Diag(\Gamma)$ is contained in the centralizer of 
$\sigma$ in 
$T_e$, which is $T_2\times\langle\sigma\rangle$ by Lemma \ref{le:phiH-id}.(3).
The result follows.
\end{proof}

From now on, fix a Cartan subalgebra $\cH$ of $\cL$ and an automorphism 
$\sigma\in\Aut(\cL)$ such that $\sigma\vert_\cH=-\id$. Note that $\sigma$ is unique up to conjugation because of Lemma \ref{le:phiH-id}.(1). Denote by $Q$ the root lattice $Q\bydef\ZZ\Phi=\ZZ\Delta$ (notation as above). The torus $T$ is naturally isomorphic to the group of characters of $Q$, and this isomorphism restricts to 
a group isomorphism $T_2\simeq \Hom(Q/2Q,\{\pm 1\})$. Any 
$\chi\in\Hom(Q/2Q,\{\pm 1\})$ corresponds to the automorphism $\tau_\chi$ whose
restriction to $\cL_\alpha$ is $\chi(\alpha +2Q)\id$ for all $\alpha\in\Phi$.
In particular, any element of $T_2$ acts a $\pm\id$ on 
the two-dimensional subspace $\cL_\alpha+\cL_{-\alpha}$ for any $\alpha\in\Phi^+$.

This gives a bijection:
\begin{equation}\label{eq:T2}
\begin{split}
\{\text{subgroups of $T_2$}\}&\longrightarrow 
   \{\text{subgroups of $Q$ containing $2Q$}\}\\
   S\quad&\mapsto\quad E\ 
        \text{such that $E/2Q=\bigcap_{\tau_\chi\in S}\ker\chi$}\,.
\end{split}
\end{equation}
In the reverse direction, a subgroup $E$ with $2Q\leq E\leq Q$ corresponds to
the subgroup $T_E\bydef\{\tau_\chi\mid \chi(E/2Q)=1\}$. 

For any positive root $\alpha\in\Phi^+$, pick a nonzero element 
$x_\alpha\in\cL_\alpha$. Let $E$ be a subgroup with $2Q\leq E\leq Q$ and denote by 
$\overline{Q}_E$ the 
$2$-elementary group $Q/E\times \ZZ/2$. Define the $\overline{Q}_E$-grading 
$\Gamma_E$ on $\cL$ as follows:
\begin{equation}\label{eq:GammaE}
\begin{split}
\cL_{(q+E,\bar 0)}
 &=\bigoplus_{\alpha\in\Phi^+\cap(q+E)}\FF\left(x_\alpha+\sigma(x_\alpha)\right),\\
\cL_{(q+E,\bar 1)}
  &=\begin{cases}
   \bigoplus_{\alpha\in\Phi^+\cap(q+E)}\FF\left(x_\alpha-\sigma(x_\alpha)\right)
       &\text{if $q+E\neq E$,}\\
  \cH\oplus\left(\bigoplus_{\alpha\in\Phi^+\cap(q+E)}
           \FF\left(x_\alpha-\sigma(x_\alpha)\right)\right)
           &\text{if $q+E=E$.}
        \end{cases}
\end{split}
\end{equation}

\begin{remark}\label{re:GammaE}
Note that the homogeneous spaces of $\Gamma_E$ are the common eigenspaces for 
the action of 
the group $T_E\times\langle\sigma\rangle$.
\end{remark}

The extreme cases are $\Gamma_Q$, which is a grading by $\ZZ/2$ with 
$\cL\subo=\{x\in\cL\mid \sigma(x)=x\}$ and 
$\cL\subuno=\{x\in\cL\mid \sigma(x)=-x\}$, and $\Gamma_{2Q}$, which is a grading
by $Q/2Q\times\ZZ/2\simeq \left(\ZZ/2\right)^{n+1}$.


For any $E$ as above, define the subgroup $E^\circ$ as follows:
\begin{equation}\label{eq:Ecirc}
E^\circ\bydef 2Q +\ZZ(\Phi^+\cap E)+\ZZ\{\alpha-\beta\mid \alpha,\beta\in\Phi^+\ 
\text{and}\ \alpha-\beta\in E\}.
\end{equation}
Note that $2Q\leq E^\circ\leq E\leq Q$.

\begin{remark}\label{re:Ecirc}
Another way to define $E^\circ$ is as the subgroup of $E$ generated by $2Q$, 
by the elements $\alpha\in\Phi^+$ such that 
$x_\alpha-\sigma(x_\alpha)$ is 
in the homogeneous component of $\Gamma_E$ that contains 
$\cH$, and by the elements
$\alpha-\beta$ for $\alpha,\beta\in\Phi^+$ such that 
$x_\alpha-\sigma(x_\alpha)$
and $x_\beta-\sigma(x_\beta)$ are in the same homogeneous component of $\Gamma_E$.
\end{remark}

\begin{proposition}\label{pr:DiagGE}
Under the conditions above, the gradings $\Gamma_E$ and $\Gamma_{E^\circ}$ are equivalent and $\Diag(\Gamma_E)=\Diag(\Gamma_{E^\circ})$ equals 
$T_{E^\circ}\times\langle\sigma\rangle$.
\end{proposition}
\begin{proof}
The homogeneous components of $\Gamma_E$ and $\Gamma_{E^\circ}$ coincide, so they are trivially equivalent.

Proposition \ref{pr:pure} gives $\Diag(\Gamma_E)=\left(\Diag(\Gamma_E)\cap T_2\right)\times \langle \sigma\rangle$. For any character $\chi\in\Hom(Q/2Q,\{\pm 1\})$, 
let $\tau_\chi$ be the associated element in $T_2$. For any $\alpha\in\Phi^+$, 
$\tau_\chi\left(x_\alpha\pm\sigma(x_\alpha)\right)
=\chi(\alpha+2Q)\left(x_\alpha\pm\sigma(x_\alpha)\right)$.
Then $\tau_\chi$ belongs to $\Diag(\Gamma_E)$ if and only if $\chi(\alpha+2Q)=1$
for any $\alpha\in\Phi^+\cap E$ (because $\tau_\chi\vert_\cH=\id$), and
$\chi(\alpha+2Q)=\chi(\beta+2Q)$ for any $\alpha,\beta\in\Phi^+$ with 
$\alpha-\beta\in E$. In other words, $\tau_\chi$ lies in $\Diag(\Gamma_E)$ if
and only if $E^\circ/2Q$ is contained in $\ker\chi$, if and only if $\tau_\chi$ lies in $T_{E^\circ}$.
\end{proof}

Let us show now that $E^\circ$ may be strictly contained in $E$.

\begin{example}\label{ex:D6}
Let $\cL$ be the orthogonal Lie algebra $\frso(V,\bup)$ of a vector space $V$ of
dimension $12$, endowed with a nondegenerate symmetric bilinear form $\bup$. 
Pick a basis $\{u_1,\ldots,u_6,v_1,\ldots,v_6\}$ of $V$ with 
$\bup(u_i,u_j)=0=\bup(v_i,v_j)$ and $\bup(u_i,v_j)=\delta_{ij}$ (Kronecker delta) for all $i,j$. 
The diagonal elements of $\frso(V,\bup)$ relative to this basis form a Cartan subalgebra $\cH$, and the weights of $V$ relative to $\cH$ are 
$\pm\veps_1,\ldots,\pm\veps_6$, where 
$h.u_i=\veps_i(h)u_i$, $h.v_i=-\veps_i(h)v_i$ for all $h\in\cH$ and $i=1,\ldots,6$. The set of roots is $\Phi=\{\pm\veps_i\pm\veps_j\mid i\neq j\}$. 
Up to a scalar, the nondegenerate bilinear form on $\cH^*$ induced 
by the Killing form is given by $(\veps_i\vert\veps_j)=\delta_{ij}$ (Kronecker delta).

As a system of simple roots
take $\Delta=\{\veps_1-\veps_2,\veps_2-\veps_3,\veps_3-\veps_4,\veps_4-\veps_5,
\veps_5-\veps_6,\veps_5+\veps_6\}$. The highest weight of one of the half-spin
modules es $\frac{1}{2}(\veps_1+\cdots+\veps_6)$. Denote by 
$\Lambda'$ the set 
of  weights of this half spin module: 
$\Lambda'=\bigl\{\frac12(\pm\veps_1\pm\cdots\pm\veps_6)\mid
\text{even number of $+$ signs}\bigr\}$. Write $W'=\ZZ\Lambda'$. Then $E=2W'$ satisfies 
$2Q\leq E=
2Q+\ZZ(\veps_1+\cdots+\veps_6)\leq Q$.

It follows at once that for any $i<j$, $r<s$ and $(i,j)\neq (r,s)$, the element
$\frac12\bigl((\veps_i\pm\veps_j)-(\veps_r\pm\veps_s)\bigr)$ does not belong neither to $Q$, nor to $\frac12(\veps_1+\cdots\veps_6)+Q$. Hence we get that for any two different positive roots 
$\alpha,\beta\in\Phi^+$, we get $\alpha-\beta\not\in E$. It is also clear that
$\alpha$ does not belong to $E$ for any $\alpha\in\Phi^+$, so \eqref{eq:Ecirc}
shows that $E^\circ$ equals $2Q$.
\end{example}

\begin{remark}\label{re:caution}
A word of caution is in order here. The notion of equivalence in 
\cite{Hes} is more restrictive. It turns out that
 $\Gamma_E$ and 
$\Gamma_{E^\circ}$ are not equivalent in this more restrictive sense  if $E^\circ$ is contained properly in $E$, as in Example \ref{ex:D6}, even though their
homogeneous components coincide, so they are trivially equivalent in our sense.
\end{remark}

\smallskip

The situation is different for $D_4$, and this will be crucial in our examples of non-group gradings in the next sections.

\begin{example}\label{ex:D4}
Let $\cL$ be the orthogonal Lie algebra $\frso(V,\bup)$ of a vector space $V$ of
dimension $8$, endowed with a nondegenerate symmetric bilinear form 
$\bup$. As in Example \ref{ex:D6},
pick a basis $\{u_1,\ldots,u_4,v_1,\ldots,v_4\}$ of $V$ with 
$\bup(u_i,u_j)=0=\bup(v_i,v_j)$ and $\bup(u_i,v_j)=\delta_{ij}$ for all $i,j$, and consider the corresponding Cartan subalgebra of diagonal elements relative to this basis, and the  system of simple roots
$\Delta=\{\veps_1-\veps_2,\veps_2-\veps_3,\veps_3-\veps_4,
\veps_3+\veps_4\}$. Pick the highest weight 
$\frac12(\veps_1+\veps_2+\veps_3+\veps_4)$
of one of the half-spin modules, and let $\Lambda'$ be its set of weights. Finally, write $W'=\ZZ\Lambda'$ and consider the subgroup 
$E=2W'=2Q+\ZZ(\veps_1+\veps_2+\veps_3+\veps_4)$, 
whose index $[E:2Q]$ is $2$.

For any $1\leq r<s\leq 4$, let $\bar r<\bar s$ be such that
 $\{r,s,\bar r,\bar s\}=\{1,2,3,4\}$. Then the positive roots 
$\veps_r+\veps_s$ and $\veps_{\bar r}+\veps_{\bar s}$ 
satisfy that its difference belongs to $E$, and the same happens with
$\veps_1+\veps_2+\veps_3+\veps_4=(\veps_r+\veps_s)-(\veps_{\bar r}+\veps_{\bar s})+2(\veps_{\bar r}+\veps_{\bar s})$. Therefore in this case we get $E=E^\circ$.

Proposition \ref{pr:DiagGE} shows that the diagonal group of the
grading $\Gamma_E$ is $T_E\times\langle\sigma\rangle$, which is 
isomorphic to $\left(\ZZ/2\right)^4$.
\end{example}

Our last example is Panayuk's example mentioned in Section \ref{se:intro}.

\begin{example}\label{ex:Dn}
Let $\cL$ be the orthogonal Lie algebra $\frso(V,\bup)$ of a vector space $V$ of
dimension $2n$, endowed with a nondegenerate symmetric bilinear form $\bup$. 
Pick a basis $\{u_1,\ldots,u_n,v_1,\ldots,v_n\}$ of $V$ with 
$\bup(u_i,u_j)=0=\bup(v_i,v_j)$ and $\bup(u_i,v_j)=\delta_{ij}$ for all 
$i,j$ as in Example \ref{ex:D6}, and use the notation there. Let 
$W=\ZZ\veps_1\oplus\cdots\oplus\ZZ\veps_n$ be the $\ZZ$-span of the 
weights of the natural module $V$. Note that $W=Q+\ZZ\veps_1$ and 
$2\veps_1=(\veps_1-\veps_2)+(\veps_1+\veps_2)$ belongs to $Q$.
Hence the index $[W:Q]$ is $2$. Consider the subgroup $E=2W$, which lies between $2Q$ and $Q$, and the associated grading 
$\Gamma_{2W}$ over $Q/(2W)\times \ZZ/2\simeq \left(\ZZ/2\right)^n$
in \eqref{eq:GammaE}. 

The intersection $\Phi^+\cap 2W$ is empty, and the only pairs of positive roots that belong to the same coset modulo $2W$ are
$\veps_i-\veps_j$ and $\veps_i+\veps_j$ for $1\leq i< j\leq n$.
It follows that $E^\circ=E$ and that the homogeneous component
$\cL_{(E,\bar 1)}$ coincides with the Cartan subalgebra, while all the other nonzero homogeneous components have dimension two.
\end{example}

\begin{remark}\label{re:triality}
The same properties are valid for the grading in Example \ref{ex:D4}. But for $n=4$, the natural and half-spin modules are related by triality, so the grading in Example \ref{ex:D4} is equivalent to the one in 
Example \ref{ex:Dn} for $n=4$. However, the specific $E$ in
Example \ref{ex:D4} will be crucial in the next section (proof of Proposition \ref{pr:set_grading}).
\end{remark}

\smallskip

\section{A non-group grading on $D_{13}$}\label{se:D13}

We may number the points in the projective plane $\PP^2(\FF_3)$
over the field $\FF_3$ of three elements from 
$1$ to $13$, so that the lines in $\PP^2(\FF_3)$ are the ones consisting of the points:
\begin{equation}\label{eq:lines}
\begin{matrix}
\{1,2,3,4\}&\{2,5,8,11\}&\{3,5,9,13\}&\{4,5,10,12\}\\
\{1,5,6,7\}&\{2,6,9,12\}&\{3,6,10,11\}&\{4,6,8,13\}\\
\{1,8,9,10\}&\{2,7,10,13\}&\{3,7,8,12\}&\{4,7,9,11\}\\
\{1,11,12,13\}&&&
\end{matrix}
\end{equation}
(The reader can check that any two points lie in a unique line, and any two lines intersect in a unique point.) Denote by $\mathfrak{L}$ this set of lines.

As in Examples \ref{ex:D6} and \ref{ex:D4}, let $\cL$ be the orthogonal Lie algebra $\frso(V,\bup)$ of a vector space $V$ of
dimension $26$, endowed with a nondegenerate symmetric bilinear form $\bup$. 
Pick a basis $\{u_1,\ldots,u_{13},v_1,\ldots,v_{13}\}$ of $V$ with 
$\bup(u_i,u_j)=0=\bup(v_i,v_j)$ and $\bup(u_i,v_j)=\delta_{ij}$ for all $i,j$. 
The diagonal elements of $\frso(V,\bup)$ relative to this basis form a Cartan subalgebra $\cH$, and the weights of $V$ relative to $\cH$ are 
$\pm\veps_1,\ldots,\pm\veps_{13}$, where 
$h.u_i=\veps_i(h)u_i$, $h.v_i=-\veps_i(h)v_i$ for all $h\in\cH$ and 
$i=1,\ldots,13$. The set of roots is 
$\Phi=\{\veps_i\pm\veps_j\mid i\neq j\}$. As a system of simple roots
take $\Delta=\{\veps_1-\veps_2,\veps_2-\veps_3,\veps_3-\veps_4,\dots,
\veps_{12}-\veps_{13},\veps_{12}+\veps_{13}\}$.

Fix an automorphism $\sigma$ of $\cL$ with $\sigma\vert_\cH=-\id$, and for any positive root $\alpha\in\Phi^+$ pick a nonzero element
$x_\alpha\in\cL_\alpha$.

For any line $\ell=\{i,j,k,l\}\in\frL$, with $i<j<k<l$ consider the set
$\frP_{\ell}$ whose elements are the following six subsets of pairs of orthogonal positive  roots
\begin{equation}\label{eq:six}
\begin{matrix}
\{\veps_i+\veps_j,\veps_k+\veps_l\},&\ 
\{\veps_i+\veps_k,\veps_j+\veps_l\},&\ 
\{\veps_i+\veps_l,\veps_j+\veps_k\},\\
\{\veps_i-\veps_j,\veps_k-\veps_l\},&\ 
\{\veps_i-\veps_k,\veps_j-\veps_l\},&\ 
\{\veps_i-\veps_l,\veps_j-\veps_k\}.
\end{matrix}
\end{equation}
Denote by $\frP$ the union $\frP=\bigcup_{\ell\in\frL}\frP_{\ell}$. Its
elements are subsets consisting of a pair  of orthogonal positive roots. Hence 
$\frP$ contains $13\times 6=78$ elements, and each positive root appears in exactly one of the pairs in $\frP$. Then $\cL$ decomposes as
\begin{multline}\label{eq:D13}
\cL=\cH\oplus\Bigl(\bigoplus_{\{\alpha,\beta\}\in\frP}
       \bigl(\FF(x_\alpha+\sigma(x_\alpha))
          +\FF(x_\beta+\sigma(x_\beta))\bigr)\Bigr) \\\oplus
      \Bigl(\bigoplus_{\{\alpha,\beta\}\in\frP}
       \bigl(\FF(x_\alpha-\sigma(x_\alpha))
          +\FF(x_\beta-\sigma(x_\beta))\bigr)\Bigr).
\end{multline}
Hence $\cL$  decomposes into the direct sum of the Cartan subalgebra and the direct sum of $13\times 6\times 2=156$ two-dimensional abelian subalgebras. The reader may recall here Example \ref{ex:Dn}, where a group-grading with these properties is highlighted.

\begin{proposition}\label{pr:set_grading}
The decomposition in Equation \eqref{eq:D13} is a set grading of 
$\cL$.
\end{proposition}
\begin{proof}
For any $h\in\cH$ and $\alpha\in\Phi^+$, 
$[h,x_\alpha\pm\sigma(x_\alpha)]
=\alpha(h)(x_\alpha\mp\sigma(x_\alpha))$. Also, for any two orthogonal positive roots $\alpha,\beta$, 
$[x_\alpha\pm\sigma(x_\alpha),x_\beta\pm\sigma(x_\beta)]=0$, because $\alpha\pm\beta$ is not a root. We must check that the bracket
of two subspaces of the form  $\FF(x_\alpha\pm\sigma(x_\alpha))+\FF(x_\beta\pm\sigma(x_\beta))$, with $\{\alpha,\beta\}\in\frP$, is contained in another subspace of this form. So let us take two such elements $\{\alpha,\beta\}\in\frP_{\ell_1}$ and $\{\alpha',\beta'\}\in\frP_{\ell_2}$. There are two different possibilities:
\begin{itemize}
\item 
$\ell_1\cap\ell_2$ consists of a single point: $\ell_1=\{i,j,k,l\}$, 
$\ell_2=\{i,p,q,r\}$ for distinct $i,j,k,l,p,q,r$. This shows that, up to reordering, $(\alpha\vert\alpha')\neq 0$, but $(\alpha\vert\beta')=0=
(\alpha'\vert\beta)$. This is because, up to reordering, the roots involved are of the form $\alpha=\pm\veps_i\pm\veps_j$, 
$\beta=\pm\veps_k\pm\veps_l$, $\alpha'=\pm\veps_i\pm\veps_p$, 
$\beta=\pm\veps_q\pm\veps_r$. Since $(\alpha\vert\alpha')\neq 0$, 
either $\alpha+\alpha'$ or $\alpha-\alpha'$ is a root (but not both), while 
$\alpha\pm\beta'$ and $\alpha'\pm\beta$ are not roots. 
It follows that the bracket of $\FF(x_\alpha\pm\sigma(x_\alpha))+
\FF(x_\beta\pm\sigma(x_\beta))$ with 
$\FF(x_{\alpha'}\pm\sigma(x_{\alpha'}))
+\FF(x_{\beta'}\pm\sigma(x_{\beta'}))$ equals
$\FF[x_\alpha\pm\sigma(x_\alpha),x_{\alpha'}\pm\sigma(x_{\alpha'})]$,
which equals 
$\FF\left(x_{\alpha\pm\alpha'}\pm\sigma(x_{\alpha\pm\alpha'})\right)$, depending on whether $\alpha+\alpha'$ or $\alpha-\alpha'$ is a root.

\item 
$\ell_1=\ell_2$. Without loss of generality we may assume that
$\ell=\{1,2,3,4\}$. Then the subalgebra generated by $x_\alpha$ and
$\sigma(x_\alpha)$ for $\alpha$ in the six subsets associated to $\ell$ generate a subalgebra of $\cL$ isomorphic to $D_4$. The decomposition induced on this subalgebra from the decomposition in \eqref{eq:D13} is exactly the group-grading $\Gamma_E$ in Example \ref{ex:D4}, and hence the bracket of $\FF(x_\alpha\pm\sigma(x_\alpha))+
\FF(x_\beta\pm\sigma(x_\beta))$ with 
$\FF(x_{\alpha'}\pm\sigma(x_{\alpha'}))
+\FF(x_{\beta'}\pm\sigma(x_{\beta'}))$ is contained in another 
`homogeneous component' in \eqref{eq:D13}.
\end{itemize}
We conclude that the decomposition in \eqref{eq:D13} is indeed a set grading on 
$\cL$.
\end{proof}

Denote the grading given by \eqref{eq:D13} by $\Gamma$. Our aim now is to show that $\Gamma$ cannot be realized as a group-grading. We will do it by computing its diagonal group.

To begin with, note that the automorphism $\sigma$ belongs to 
$\Diag(\Gamma)$, and that any element of $\Diag(\Gamma)$ must act as a scalar on 
$\cH$. The arguments in the proof of 
Proposition \ref{pr:pure} work for $\Gamma$ and hence Equation 
\eqref{eq:DiagG} holds: $\Diag(\Gamma)
=\left(\Diag(\Gamma)\cap T_2\right)\times\langle\sigma\rangle$. 

For any $\chi\in\Hom(Q/2Q,\{\pm 1\})$, let $\tau_{\chi}$ be the corresponding element in $T_2$ (recall Equation \eqref{eq:T2}). Then, as in the proof of 
Proposition \ref{pr:DiagGE}, $\tau_\chi$ is in 
$\Diag(\Gamma)$ if and only if $\chi(\alpha+2Q)=\chi(\beta +2Q)$ for all
$\{\alpha,\beta\}\in\frP$, if and only if $\bigl\{(\alpha-\beta)+2Q\mid \{\alpha,\beta\}\in\frP\bigr\}$ is contained in $\ker\chi$. Therefore we get
$\Diag(\Gamma)=T_E\times\langle\sigma\rangle$ for 
$E=2Q+\ZZ\bigl\{\alpha-\beta\mid
\{\alpha,\beta\}\in\frP\bigr\}$.

In particular, for $\ell=\{i,j,k,l\}\in\frL$, with $1\leq i<j<k<l\leq 13$, the sum
$\veps_i+\veps_j+\veps_k+\veps_l=(\veps_i+\veps_j)-(\veps_k+\veps_l)+2(\veps_k+\veps_l)$ lies in $E$.
Considering the lines in the first column of \eqref{eq:lines}, and adding the elements 
$\veps_i+\veps_j+\veps_k+\veps_l$  for these lines, we check that 
$E$ contains the elements $4\veps_1+\sum_{i=2}^{13}\veps_i$. As
$4\veps_1=2\left((\veps_1-\veps_2)+(\veps_1+\veps_2)\right)$ lies in 
$2Q\leq E$, it follows that $\sum_{i\neq 1}\veps_i$ lies in $E$. 
In the same vein $\sum_{i\neq 2}\veps_i$ lies in $E$, and 
subtracting these two elements, we see that $\veps_1-\veps_2$ lies
 in $E$. The same argument shows that all the roots of the form 
$\veps_i-\veps_j$ lie in $E$. But then we get 
$Q=\ZZ\Delta\subseteq E+\ZZ(\veps_{12}+\veps_{13})\subseteq Q$, so $Q=E+\ZZ(\veps_{12}+\veps_{13})$. Also we have
$2(\veps_{12}+\veps_{13})\in 2Q\subseteq E$, and hence the index of
$E$ in $Q$ is at most $2$, and  
$\Diag(\Gamma)=T_E\times\langle\sigma\rangle$ is isomorphic to either $C_2$ or $C_2\times C_2$. Actually, the splitting 
$V=(\FF u_1+\cdots+\FF u_6)\oplus (\FF v_1+\cdots+\FF v_6)$ gives a 
$\ZZ/2$-grading on $\cL$, where the even (respectively odd) part 
consists of the elements in $\cL$ that preserve (respectively swap) the 
subspaces $\FF u_1+\cdots+\FF u_6$ and $\FF v_1+\cdots+\FF v_6$. 
The order two  automorphism $\tau$ that fixes the even part and is 
$-\id$ on the odd part lies in $T$. It fixes the root spaces 
$\cL_{\veps_i-\veps_j}$ and is $-\id$ on the root spaces 
$\cL_{\veps_i+\veps_j}$. It
follows that $\tau$ is in $\Diag(\Gamma)\cap T_2$, and we conclude that the diagonal group of $\Gamma$ is exactly
$\Diag(\Gamma)=\langle\tau,\sigma\rangle$. 

\smallskip

Alternatively, we can use some software, for instance SageMath
\cite{Sage}, to compute $[Q:E]$. Let
$W$ be the group generated by $\veps_1,\ldots,\veps_{13}$. (Note  
$[W:Q]=2$.) 
Write in a SageMath cell the following:

\smallskip

{\obeylines
\def\salto{\noindent\phantom{A=matrix([}}
\noindent\texttt{%
A=matrix([[2,-2,0,0,0,0,0,0,0,0,0,0,0],[0,2,-2,0,0,0,0,0,0,0,0,0,0],
\salto[0,0,2,-2,0,0,0,0,0,0,0,0,0],[0,0,0,2,-2,0,0,0,0,0,0,0,0],
\salto[0,0,0,0,2,-2,0,0,0,0,0,0,0],[0,0,0,0,0,2,-2,0,0,0,0,0,0],
\salto[0,0,0,0,0,0,2,-2,0,0,0,0,0],[0,0,0,0,0,0,0,2,-2,0,0,0,0],
\salto[0,0,0,0,0,0,0,0,2,-2,0,0,0],[0,0,0,0,0,0,0,0,0,2,-2,0,0],
\salto[0,0,0,0,0,0,0,0,0,0,2,-2,0],[0,0,0,0,0,0,0,0,0,0,0,2,-2],
\salto[0,0,0,0,0,0,0,0,0,0,0,2,2],
\salto[1,1,1,1,0,0,0,0,0,0,0,0,0],[1,0,0,0,1,1,1,0,0,0,0,0,0],
\salto[1,0,0,0,0,0,0,1,1,1,0,0,0],[1,0,0,0,0,0,0,0,0,0,1,1,1],
\salto[0,1,0,0,1,0,0,1,0,0,1,0,0],[0,1,0,0,0,1,0,0,1,0,0,1,0],
\salto[0,1,0,0,0,0,1,0,0,1,0,0,1],[0,0,1,0,1,0,0,0,1,0,0,0,1],
\salto[0,0,1,0,0,1,0,0,0,1,1,0,0],[0,0,1,0,0,0,1,1,0,0,0,1,0],
\salto[0,0,0,1,1,0,0,0,0,1,0,1,0],[0,0,0,1,0,1,0,1,0,0,0,0,1],
\salto[0,0,0,1,0,0,1,0,1,0,1,0,0]]);
\noindent A.elementary\_divisors()
}}
\smallskip

\noindent where the first six rows correspond to the coordinates of the elements of the basis $2\Delta$ of $2Q$ in the natural basis $\{\veps_1,\ldots,\veps_{13}\}$ of $W$. 
The last six rows are the coordinates of the elements 
$\veps_i+\veps_j+\veps_k+\veps_l$ for each line $\{i,j,k,l\}\in\frL$.

The outcome of running the SageMath cell above is the following:
\smallskip

\noindent\texttt{%
[1,1,1,1,1,1,1,1,1,1,1,1,4,0,0,0,0,0,0,0,0,0,0,0,0,0]}
\smallskip

\noindent giving the elementary divisors of the matrix $A$, 
and this shows that the index $[W:E]$ is $4$, and hence the index $[Q:E]$ is $2$. 

\smallskip

We conclude that the universal group of $\Gamma$ is also isomorphic to $C_2\times C_2$, because its group of characters is
$\Diag(\Gamma)\simeq C_2\times C_2$, but  there are $157$ nonzero 
homogeneous components, and hence the natural map $\iota:S\rightarrow U(\Gamma)$ cannot be one-to-one, where $S$ denotes the set of homogeneous components of 
$\Gamma$. We have proved the following result:

\begin{theorem}\label{th:D13}
The set grading $\Gamma$ of the simple Lie algebra of type $D_{13}$ in \eqref{eq:D13} cannot be realized as a group grading.
\end{theorem}

Actually, this grading makes sense over an arbitrary field $\FF$ of characteristic not two. Indeed, let $\cL$ be the orthogonal Lie algebra 
$\frso(V,\bup)$ of a vector space $V$ over $\FF$ of
dimension $26$, endowed with a nondegenerate symmetric bilinear form $\bup$ of maximal Witt index. That is, $\cL$ is the classical split simple
Lie algebra of type $D_{13}$ over $\FF$.
Pick a basis $\{u_1,\ldots,u_{13},v_1,\ldots,v_{13}\}$ of $V$ with 
$\bup(u_i,u_j)=0=\bup(v_i,v_j)$ and $\bup(u_i,v_j)=\delta_{ij}$ for all $i,j$. 
As before, the diagonal elements of $\frso(V,\bup)$ relative to this basis form a Cartan subalgebra $\cH$, and the weights (elements of the dual
$\cH^*$) of $V$ relative to $\cH$ are 
$\pm\veps_1,\ldots,\pm\veps_{13}$, where 
$h.u_i=\veps_i(h)u_i$, $h.v_i=-\veps_i(h)v_i$ for all $h\in\cH$ and $i=1,\ldots,n$. 
Up to a scalar, the nondegenerate bilinear form on $\cH^*$ induced 
from the trace form on $V$ is given by 
$(\veps_i\vert\veps_j)=\delta_{ij}$ (Kronecker delta).

Identify any $x\in \cL$ with its coordinate matrix relative to the basis above, which has the following block form 
$
\begin{pmatrix}A&B\\ C&-A^t\end{pmatrix}
$,
where the blocks are $13\times 13$ matrices, with $B$ and $C$ symmetric: $B=B^t$, $C=C^t$.

The root space decomposition of $\cL$ relative to $\cH$ is
\begin{equation}\label{eq:root}
\cL=\cH\oplus\left(\bigoplus_{i\neq j}\cL_{\veps_i-\veps_j}\right)
\oplus\left(\bigoplus_{i<j}\cL_{\veps_i+\veps_j}\right)
\oplus\left(\bigoplus_{i<j}\cL_{-\veps_i-\veps_j}\right)
\end{equation}
where the roots spaces are:
\[
\cL_{\veps_i-\veps_j}
=\FF\begin{pmatrix} E_{ij}&0\\ 0&-E_{ji}\end{pmatrix},\quad
\cL_{\veps_i+\veps_j}
=\FF\begin{pmatrix}0& E_{ij}\\ 0&0\end{pmatrix},\quad
\cL_{-\veps_i-\veps_j}
=\FF\begin{pmatrix} 0&0\\ E_{ij}&0\end{pmatrix},
\]
where $E_{ij}$ is the $13\times 13$-matrix with $1$ in the 
$(i,j)$-position and $0$'s elsewhere.

Let $W$ be the free abelian group with generators 
$\epsilon_1,\ldots,\epsilon_{13}$
 (note the slightly different notation:
$\epsilon_i$ for free generators, and $\veps_i$ for weights), and its
index $2$ subgroup $Q$ freely generated by $\epsilon_1-\epsilon_2,
\epsilon_2-\epsilon_3,\ldots,\epsilon_{12}-\epsilon_{13},
\epsilon_{12}+\epsilon_{13}$.

The root space decomposition \eqref{eq:root} is a $Q$-grading, with
$\cL_{\pm\epsilon_i\pm\epsilon_j}\bydef\cL_{\pm\veps_i\pm\veps_j}$ for all 
$i,j$. The automorphism $\sigma: x\mapsto -x^t$ has order $2$ and satisfies $\sigma(\cL_\alpha)=\cL_{-\alpha}$ for all roots $\alpha$.
This automorphism $\sigma$ can be used to define the grading 
$\Gamma$ as in \eqref{eq:D13}

The only scalar multiples of a root $\alpha$ are $\pm\alpha$, hence the argument in the proof of Proposition \ref{pr:pure} gives that any automorphism of $\cL$ that acts as a scalar multiple of $\id$ on $\cH$
must act necessarily as $\id$ or $-\id$. Now the arguments in 
Lemma \ref{le:phiH-id}
and Proposition \ref{pr:pure} apply to show that $T_2$ is the centralizer of $\sigma$ in $T$ and that \eqref{eq:DiagG} holds here.

On the other hand, any $\varphi\in T_2$ is determined by its
restriction to the root spaces $\cL_{\veps_1-\veps_2},\ldots,
\cL_{\veps_{12}-\veps_{13}}\cL_{\veps_{12}+\veps_{13}}$, and this allows us to identify $T_2$ with the group $\Hom(Q/2Q,\{\pm1\})$.

Proposition \ref{pr:DiagGE} and its proof thus remain valid over an arbitrary field of characteristic not two and, therefore, we conclude that the diagonal
group $\Diag(\Gamma)$ is  isomorphic to 
$C_2\times C_2$. However, over arbitrary fields the diagonal group is,
up to isomorphism, the group of characters of the universal grading group, but not conversely.

However, for any $\{\alpha,\beta\}\in\frP$ consider the subalgebra 
$\cH\oplus\bigl(\FF(x_\alpha+\sigma(x_\alpha))+
\FF(x_\beta+\sigma(x_\beta))\bigr) \oplus
\bigl(\FF(x_\alpha-\sigma(x_\alpha))+
\FF(x_\beta-\sigma(x_\beta))\bigr)$. It consists of three of the homogeneous components of $\Gamma$. The bracket of any two
of these components is nonzero and lies in the other component, because $\alpha\pm\beta$ is not a root.
This shows that all the generators of the universal group have order at most two and, therefore, the universal group is $2$-elementary. But the characteristic of 
$\FF$ being not two,
the group of characters of a $2$-elementary group is isomorphic to itself.
Hence the universal group of $\Gamma$ is isomorphic to $C_2\times C_2$, generated by $\sigma$ and by the order two automorphism 
$\tau$ that fixes the elements of the form 
$\begin{pmatrix}A&0\\ 0&-A^t\end{pmatrix}$ and multiplies by $-1$ the
elements of the form $\begin{pmatrix}0&B\\ C&0\end{pmatrix}$.  We thus get the same contradiction leading to Theorem \ref{th:D13}, which is then extended as follows.

\begin{theorem}\label{th:D13_2}
The set grading $\Gamma$ in \eqref{eq:D13} of the classical split simple Lie algebra of type $D_{13}$, over an arbitrary field of characteristic not two, cannot be realized as a group grading.
\end{theorem}

\smallskip

\section{An infinite family of non-group gradings on orthogonal Lie algebras}%
\label{se:Steiner}

A careful look at the non-group grading in Section \ref{se:D13} shows
that a key point is the existence of the set of lines in \eqref{eq:lines}.

All the arguments in the previous section can be applied as long as $n>4$ and there is a set $\frL$ of subsets of $4$ elements of 
$\{1,\ldots,n\}$, that we will call
lines, such that  any
two points lie in a unique line. Any such $\frL$ is called a 
$2$-$(n,4,1)$ design, or a Steiner system of type $S(2,4,n)$. (See e.g. the survey paper
\cite{ReidRosa}.)

Actually, all the arguments in the proof of 
Proposition \ref{pr:set_grading} work, but a possibility must be added, as there are now three options for the intersection of two lines 
$\ell_1$ and $\ell_2$: either $\ell_1\cap\ell_2=\emptyset$, or 
$\ell_1\cap\ell_2$ consists of a single point, or $\ell_1=\ell_2$. The extra option:
$\ell_1\cap\ell_2=\emptyset$, is dealt with easily, as in this case 
$\ell_1=\{i,j,k,l\}$ and $\ell_2=\{p,q,r,s\}$ for distinct 
$i,j,k,l,p,q,r,s$, and hence the subsets of roots $\{\alpha,\beta\}$, 
$\{\alpha',\beta'\}$ satisfy that $\alpha,\beta,\alpha',\beta'$ are all orthogonal, and hence the bracket of  
$\FF(x_\alpha\pm\sigma(x_\alpha))+
\FF(x_\beta\pm\sigma(x_\beta))$ with 
$\FF(x_{\alpha'}\pm\sigma(x_{\alpha'}))
+\FF(x_{\beta'}\pm\sigma(x_{\beta'}))$ is trivial.

Therefore, any Steiner system of type $S(2,4,n)$ allows us to define
a set grading on the simple Lie algebra of type $D_{n}$.

\smallskip

It turns out (see \cite{Hanani}) that a Steiner system of type $S(2,4,n)$ exists if and only if 
$n\equiv 1\ \text{or}\ 4\pmod{12}$. One direction is easy: if $\frL$ is 
a Steiner system of type $S(2,4,n)$, $n$ must be $\equiv 1\pmod 3$, as $1$ lies in $\frac{n-1}{3}$ lines. On the other hand, if the system has  $b$ lines,  then necessarily $4b=n\frac{n-1}{3}$, so we have 
$n(n-1)\equiv 0\pmod{12}$, and we get $n\equiv 1,4\pmod{12}$.

The lowest such $n$ is $13$, where we get the Steiner system in \eqref{eq:lines}.

If $n$ is congruent to $1$ modulo $12$, and we take the corresponding set grading using \eqref{eq:six} and \eqref{eq:D13}, and compute 
$E$ as in the previous section, we check that $E$ contains the
element $\frac{n-1}{3}\veps_1+\sum_{i\neq 1}\veps_i$. But 
$\frac{n-1}{3}$ is a multiple of $4$ so, as in the previous section,
we see that $\sum_{i\neq 1}\veps_i$ lies in $E$, and all the subsequent arguments work.

As a consequence, we obtain our last result:

\begin{theorem}
For any $n\equiv 1\pmod{12}$, $n>1$, there is a 
non-group grading on the classical split simple Lie algebra of  type $D_{n}$ over an arbitrary field $\FF$ of characteristic $\neq 2$. One of the homogeneous components of this grading is a Cartan subalgebra, and all the other homogeneous components have dimension $2$.
\end{theorem}

\end{document}